\documentclass[11pt]{amsart}

\usepackage{amssymb}
\usepackage{mathtools}
\usepackage[dvipsnames]{xcolor}
\usepackage{fullpage}
\usepackage[hypertexnames=false]{hyperref}
\usepackage{aliascnt}
\usepackage{mathrsfs}

\usepackage[capitalize,nameinlink,noabbrev,nosort]{cleveref}
\hypersetup{
  colorlinks=true,
  allcolors=PineGreen
}

\theoremstyle{plain}
\newtheorem{theorem}{Theorem}[section]

\newaliascnt{proposition}{theorem}
\newtheorem{proposition}[proposition]{Proposition}
\aliascntresetthe{proposition}

\newaliascnt{conjecture}{theorem}
\newtheorem{conjecture}[conjecture]{Conjecture}
\aliascntresetthe{conjecture}
\crefname{conjecture}{Conjecture}{Conjectures}
\Crefname{conjecture}{Conjecture}{Conjectures}

\newaliascnt{lemma}{theorem}
\newtheorem{lemma}[lemma]{Lemma}
\aliascntresetthe{lemma}

\theoremstyle{definition}

\newaliascnt{definition}{theorem}
\newtheorem{definition}[definition]{Definition}
\aliascntresetthe{definition}

\crefname{definition}{Definition}{Definitions}

\theoremstyle{remark}

\newaliascnt{remark}{theorem}
\newtheorem{remark}[remark]{Remark}
\aliascntresetthe{remark}

\crefname{remark}{Remark}{Remarks}
\Crefname{remark}{Remark}{Remarks}

\numberwithin{equation}{section}

\newcommand{\scrC}{\mathscr{C}}
\newcommand{\bbC}{\mathbb{C}}
\newcommand{\bbP}{\mathbb{P}}
\newcommand{\bbQ}{\mathbb{Q}}

\begin{document}

\title{A proof of Hirose's duality conjecture}

\author{Shin-ichiro Seki}
\thanks{This research was supported by JSPS KAKENHI Grant Number JP26K06734.}
\address{Nagahama Institute of Bio-Science and Technology, 1266, Tamura, Nagahama, Shiga, 526-0829, Japan}
\email{s\_seki@nagahama-i-bio.ac.jp}
\subjclass[2020]{Primary 11M32; Secondary 05A30, 33D15}
\keywords{duality, iterated $q$-integrals, connected sums}
\begin{abstract}
Hirose formulated a duality conjecture for a $q$-discretization of iterated integrals on the four-punctured projective line, incorporating word-dependent $q$-shifts of the parameters.
In this paper, we prove this conjecture using a symmetric terminating ${}_{4}\phi_{3}$ connector.
\end{abstract}
\maketitle
\section{Introduction}
The \emph{duality} relations constitute one of the most fundamental families of relations among multiple zeta values (see \cite{Hoffman,Zagier}).
Nevertheless, duality remains a subtle and mysterious phenomenon: it is still unknown, for example, whether these relations follow from the extended double shuffle relations.

Hirose, Iwaki, Sato, and Tasaka~\cite[Theorem~1.1]{HiroseIwakiSatoTasaka} generalized this duality to iterated integrals on $\bbP^1(\bbC)\setminus\{0,1,\infty,z\}$ by considering the M\"obius transformation interchanging $0$ with $1$ and $\infty$ with $z$.
The usual duality for multiple zeta values is recovered in the limit $z\to\infty$, which corresponds to the confluence of the two punctures $z$ and $\infty$.
In \cite{HiroseSato}, Hirose and Sato developed a theory of iterated integrals on $\bbP^1(\bbC)\setminus\{0,1,\infty,z\}$ based on differentiation with respect to $z$, and introduced \emph{confluence relations} among multiple zeta values.
These relations are of particular significance, as they imply both the duality relations and the extended double shuffle relations and are conjectured to generate all $\bbQ$-linear relations among multiple zeta values.

More generally, let $A,B,C,D$ be four distinct points of $\bbP^1(\bbC)$, and consider the logarithmic differential forms
\[
\omega_{u,v}=\left(\frac{1}{t-u}-\frac{1}{t-v}\right)\mathrm{d}t
\]
for $(u,v)\in\{(A,B),(A,C),(A,D),(B,C),(B,D),(C,D)\}$.
With a suitable choice of integration path, the corresponding iterated integrals on $\bbP^1(\bbC)\setminus\{A,B,C,D\}$ also satisfy a duality arising from the M\"obius transformation interchanging $A$ with $D$ and $B$ with $C$ (see \cite[Section~1.1 and Theorem~1]{Hirose}).
This duality follows immediately from a change of variables in the iterated integrals.

Several $q$-analogues of duality are also known.
Bradley~\cite[Corollary~3]{Bradley} proved duality for $q$-multiple zeta values in the Bradley--Zhao model, and Zhao~\cite[Theorem~8.3]{Zhao} proved the corresponding result for the Schlesinger--Zudilin model.
Using a ${}_2\phi_1$ connector, Yamamoto~\cite[Theorem~3.2, the case $\epsilon=1$]{Yamamoto} proved a duality for $q$-analogues of one-variable multiple polylogarithms.
This result generalizes both of these $q$-dualities (see \cite[Remark~9]{Hirose}) as well as the duality of Hirose, Iwaki, Sato, and Tasaka for iterated integrals on $\bbP^1(\bbC)\setminus\{0,1,\infty,z\}$.

Recently, Hirose~\cite{Hirose} introduced a discretization on a $q$-lattice of iterated integrals on $\bbP^1(\bbC)\setminus\{A,B,C,D\}$ for four generic points $A,B,C,D$, using words in the six letters corresponding to
the pairs of points.
Through numerical experiments, he found suitable word-dependent
$q$-shifts of the poles and conjectured that the resulting finite sums
satisfy duality for all admissible words.
In contrast to the classical case, the underlying M\"obius transformation does not in general preserve the $q$-lattice, so the change of variables argument does not apply directly.
It is therefore remarkable that the classical duality is expected to persist at the level of finite sums, through the same operation of reversing the word and applying an involution to the six letters.
As explained in \cite[Section~3]{Hirose}, this conjecture also includes Yamamoto's duality as a special case.
We now introduce the notation needed to state Hirose's conjecture precisely.

Fix a nonnegative integer $N$.
Let $q, A, B, C, D$ be parameters.
Assume that $q, A, B, C$ are algebraically independent over $\bbQ$ and $A=q^ND$; in particular $q\neq 0$ and $q$ is not a root of unity.
We work in $\bbQ(q,A,B,C)$.
In this paper, a word is a finite sequence of the six formal letters $AB$, $AC$, $AD$, $BC$, $BD$, $CD$.
The empty word is denoted by $1$.
We write each letter in the form $uv$, where $uv$ denotes a single formal symbol rather than the product of $u$ and $v$.
By a slight abuse of notation, the symbols $u$ and $v$, when used separately, denote the corresponding parameters.
For example, when $uv=AB$, we take $u=A$ and $v=B$.

For each nonnegative integer $n$, we write $[n]\coloneqq\{1,2,\ldots,n\}$, with the convention that $[0]=\varnothing$.
For a word $w=u_1v_1\cdots u_kv_k$, we call $k$ the \emph{length} of $w$.
In particular, the empty word $1$ has length $0$.
For $i\in[k]$, set $w_{[i]}\coloneqq u_1v_1\cdots u_iv_i$ and $w^{[i]}\coloneqq u_iv_i\cdots u_kv_k$, with the convention that $w_{[0]}=w^{[k+1]}=1$.

For $uv\in\{AB,AC,AD,BC,BD,CD\}$, we use the shorthand notation
\[
\#_{uv}(w)\coloneqq \#\{i\in[k]\mid u_iv_i=uv\}.
\]
For $X\in\{A,B,C,D\}$, let $\#_{\hat{X}}(w)$ denote the number of letters of $w$ not containing $X$.

A nonempty word $w$ is \emph{admissible} if it starts with $BC$, $BD$, or $CD$ and ends with $AB$, $AC$, or $BC$.
The empty word $1$ is also regarded as admissible.
\begin{definition}[Hirose~\cite{Hirose}]
For a word $w=u_1v_1\cdots u_kv_k$ and $i\in[k]$, put
\begin{align*}
A^{(i),w}&\coloneqq Aq^{\#_{\hat{A}}(w_{[i]})},\\
B^{(i),w}&\coloneqq Bq^{\#_{\hat{A}}(w_{[i]})+\#_{AB}(w_{[i]})+\#_{CD}(w^{[i]})},\\
C^{(i),w}&\coloneqq Cq^{\#_{\hat{A}}(w_{[i]})+\#_{AC}(w_{[i]})+\#_{BD}(w^{[i]})},\\
D^{(i),w}&\coloneqq Dq^{-\#_{\hat{D}}(w^{[i]})}.
\end{align*}
The superscripts on $u_i$ and $v_i$ indicate the corresponding shifted parameters; for example, if $u_iv_i=AB$, then $u_i^{(i),w}=A^{(i),w}$ and $v_i^{(i),w}=B^{(i),w}$.

For a nonempty admissible word $w=u_1v_1\cdots u_kv_k$, we define the \emph{iterated $q$-integral} $L_q(w)$ by
\[
L_q(w)\coloneqq \sum_{0\leq n_1\leq\cdots\leq n_k\leq N}\prod_{i=1}^k\left(\frac{Aq^{-n_i}}{Aq^{-n_i}-u_i^{(i),w}}-\frac{Aq^{-n_i}}{Aq^{-n_i}-v_i^{(i),w}}\right).
\]
We also set $L_q(1)\coloneqq 1$.
\end{definition}
Admissibility ensures that none of the denominators occurring in this sum vanishes.

Let $\tau$ be the anti-involution that reverses the order of the letters and acts on individual letters by
\[
\begin{aligned}[c]
\tau(AB)&=CD,\quad & \tau(AC)&=BD,\\
\tau(CD)&=AB, & \tau(BD)&=AC,
\end{aligned}
\qquad
\tau(AD)=AD,\qquad \tau(BC)=BC.
\]
In particular, $\tau(1)=1$, and $\tau$ preserves admissibility.

In this setting, Hirose conjectured the following duality.
\begin{conjecture}[Hirose~\cite{Hirose}]\label{conj:main}
For every admissible word $w$, we have
\[
L_q(w)=L_q(\tau(w)).
\]
\end{conjecture}
Hirose proved several special cases of this conjecture in \cite[Section 4]{Hirose}.
In this paper, we prove Hirose's conjecture in full generality.
\begin{theorem}\label{thm:main}
\cref{conj:main} is true.
\end{theorem}
Our proof is based on the \emph{connected sum method}, introduced by the author and Yamamoto~\cite{SekiYamamoto} and further developed in subsequent work.
A difficulty in proving Hirose's conjecture is that the change of variables argument using a M\"obius transformation does not apply directly to the $q$-lattice.
For multiple zeta values, however, the author and Yamamoto~\cite{SekiYamamoto} gave a proof of duality that does not rely on a change of variables in iterated integrals.
Their proof uses the following connected sum.
Let $\boldsymbol{k}=(k_1,\ldots,k_r)$ and $\boldsymbol{l}=(l_1,\ldots,l_s)$ be indices of positive integers.
Assume that either both indices are nonempty, or one is empty
and the other is admissible.
Here a nonempty index is admissible if its last entry is at least $2$.
Define
\[
Z^{\mathrm{SY}}(\boldsymbol{k};\boldsymbol{l})\coloneqq\sum_{\substack{0<n_1<\cdots<n_r\\0<m_1<\cdots<m_s}}\prod_{i=1}^{r}\frac{1}{n_i^{k_i}}\cdot\frac{n_r!\,m_s!}{(n_r+m_s)!}\cdot\prod_{j=1}^{s}\frac{1}{m_j^{l_j}}.
\]
Empty products are understood to be $1$, and the summation on an
empty side consists of a single empty tuple, with endpoint $0$.
The factor
\[
\frac{n!\,m!}{(n+m)!}
\]
is called the \emph{connector}.
This factor satisfies the elementary identities
\[
\sum_{a=n+1}^{\infty}\frac{1}{a}\cdot\frac{a!\,m!}{(a+m)!}=\frac{n!\,m!}{(n+m)!}\cdot\frac{1}{m},\qquad \frac{1}{n}\cdot\frac{n!\,m!}{(n+m)!}=\sum_{b=m+1}^{\infty}\frac{n!\,b!}{(n+b)!}\cdot\frac{1}{b},
\]
where $n\geq0$ and $m\geq1$ in the first identity, and $n\geq1$ and $m\geq0$ in the second.
These identities give rise to \emph{transport relations}, which allow one unit of weight to be moved from the left index to the right:
\begin{align*}
Z^{\mathrm{SY}}(k_1,\dots, k_r,1;l_1,\dots,l_s)&=Z^{\mathrm{SY}}(k_1,\dots, k_r;l_1,\dots,l_s+1),\\
Z^{\mathrm{SY}}(k_1,\dots, k_r+1;l_1,\dots,l_s)&=Z^{\mathrm{SY}}(k_1,\dots, k_r;l_1,\dots,l_s,1).
\end{align*}
The first relation applies when $s\geq1$, and the second when $r\geq1$, provided that the connected sums are defined.
Unlike the classical change of variables argument, this procedure does not establish duality in a single step.
Repeated transport, however, moves all the weight to the right and produces the dual index, thereby proving duality.

Although the iterated $q$-integrals considered in this paper are finite sums, they are defined by discretizing iterated integrals rather than by truncating multiple series.
It is therefore natural to rewrite Seki--Yamamoto's connected sum in terms of iterated integrals, obtaining what we may call a \emph{connected iterated integral}.
For this expression, put
\[
\omega_0(t)\coloneqq\frac{\mathrm{d}t}{t},\qquad\omega_1(t)\coloneqq\frac{\mathrm{d}t}{1-t}.
\]
For nonempty indices $\boldsymbol{k}$ and $\boldsymbol{l}$, let $k\coloneqq k_1+\cdots+k_r$ and $l\coloneqq l_1+\cdots+l_s$, and write
\[
\omega_{\epsilon_1}\cdots\omega_{\epsilon_k}=\omega_1\omega_0^{k_1-1}\cdots\omega_1\omega_0^{k_r-1},\qquad\omega_{\eta_1}\cdots\omega_{\eta_l}=\omega_1\omega_0^{l_1-1}\cdots\omega_1\omega_0^{l_s-1}
\]
for the corresponding words in these differential forms.
Then the same connected sum has the integral expression
\[
Z^{\mathrm{SY}}(\boldsymbol{k};\boldsymbol{l})=\int_{\substack{0<t_1<\cdots<t_k<1\\0<s_1<\cdots<s_l<1}}\prod_{i=1}^{k}\omega_{\epsilon_i}(t_i)\cdot\mathbf{1}_{t_k+s_l<1}\cdot\prod_{j=1}^{l}\omega_{\eta_j}(s_j).
\]
Here $\mathbf{1}_{\bullet}$ is $1$ when the condition $\bullet$ holds and $0$ otherwise.
When one index is empty, its chain of integration variables is omitted and its endpoint is taken to be $0$.
Thus, in the iterated integral expression, the connector is simply the indicator function $\mathbf{1}_{t_k+s_l<1}$.
This expression follows by expanding the two iterated integrals into power series and using
\[
nm\int_{\substack{t,s>0\\t+s<1}}t^{n-1}s^{m-1}\,\mathrm{d}t\,\mathrm{d}s=\frac{n!\,m!}{(n+m)!},
\]
where $n, m\geq 1$.
The connector for the connected iterated integral satisfies the following identities, from which the transport relations will follow:
\begin{equation}\label{eq:original_trans}
\int_t^1\frac{1}{1-x}\cdot\mathbf{1}_{x+s<1} \ \mathrm{d}x=\int_s^1\mathbf{1}_{t+y<1}\cdot\frac{1}{y} \ \mathrm{d}y, \qquad \int_t^1\frac{1}{x}\cdot\mathbf{1}_{x+s<1} \ \mathrm{d}x=\int_s^1\mathbf{1}_{t+y<1}\cdot\frac{1}{1-y} \ \mathrm{d}y.
\end{equation}
These identities hold for $0<t,s<1$.
The first also holds for $t=0$ and $0<s<1$, and the second for $s=0$ and $0<t<1$.
Both follow from the substitution $x=1-y$.

Let $\boldsymbol{k}=(k_1,\ldots,k_r)$ be an admissible index, and put $k\coloneqq k_1+\cdots+k_r$.
For the corresponding word $w=BD(AC)^{k_1-1}\cdots BD(AC)^{k_r-1}$,
the specialization $A=q^N$, $B=C=\infty$, $D=1$, with $0<q<1$, gives
\[
\lim_{N\to\infty}\lim_{q\nearrow 1}\left.(1-q)^kL_q(w)\right|_{A=q^N,\;B=C=\infty,\;D=1}=\zeta(\boldsymbol{k})\coloneqq\sum_{0<n_1<\cdots<n_r}\frac{1}{n_1^{k_1}\cdots n_r^{k_r}},
\]
where the right-hand side is the multiple zeta value associated with $\boldsymbol{k}$.
Here and below, $B=C=\infty$ means that the limit
$B,C\to\infty$ is taken first, with $N$ and $q$ fixed.

With this limiting picture in mind, we return to the setting of Seki--Yamamoto's connected sum.
We construct a discrete connector $\scrC_{w,w'}(n,m)$ (\cref{def:connector}) satisfying
\[
\lim_{q\nearrow 1}\left.\scrC_{w,w'}(n,m)\right|_{A=q^N,\;B=C=\infty,\;D=1}=\mathbf{1}_{n+m\leq N}.
\]
After rescaling the lattice to $[0,1]$ and letting $N\to\infty$, the condition $n+m\leq N$ gives the connector $\mathbf{1}_{s+t<1}$ in the connected iterated integral expression of $Z^{\mathrm{SY}}(\boldsymbol{k};\boldsymbol{l})$; the boundary $s+t=1$ has measure zero.

The corresponding connected sum $Z(w,w')$ (\cref{def:connected_sum}) recovers Seki--Yamamoto's connected sum in the same limit.
That is, for
\[
w=BD(AC)^{k_1-1}\cdots BD(AC)^{k_r-1},\qquad w'=BD(AC)^{l_1-1}\cdots BD(AC)^{l_s-1},
\]
we have
\[
\lim_{N\to\infty}\lim_{q\nearrow 1}\left.(1-q)^{k+l}Z(w,w')\right|_{A=q^N,\;B=C=\infty,\;D=1}=Z^{\mathrm{SY}}(\boldsymbol{k};\boldsymbol{l}).
\]
Recovering this limit alone, however, is not sufficient.
The essential requirement is to extend \eqref{eq:original_trans} to discrete identities valid for every $N\geq0$ and general parameters satisfying $A=q^ND$.
These identities must yield transport for all six choices of the moving letter, taking each letter to its image under $\tau$ (\cref{prop:local_transport_rel}).
They must also account for the word-dependent $q$-shifts of the poles.
Accordingly, our connector depends not only on the endpoints $n, m$ but also on the words $w, w'$.

The idea of using the connected sum method when a change of variables argument does not readily apply has already been employed in proofs of duality for $q$-multiple zeta values (see, e.g., \cite{Brindle,Yamamoto}).
A striking feature of the present construction is that a single family of connectors accommodates the transport of all six letters while respecting the word-dependent $q$-shifts of the poles.
This demonstrates the flexibility of the connected sum method and suggests that it may be applicable to a broader class of discrete iterated integrals.
\section{The connector}
For positive integers $n$ and $d$, we use the $q$-shifted factorials and the shorthand notation for their products given by
\[
(a;q)_n\coloneqq\prod_{j=0}^{n-1}(1-aq^j),\qquad (a_1,\ldots,a_d;q)_n\coloneqq \prod_{\nu=1}^{d}(a_\nu;q)_n,
\]
with $(a;q)_0\coloneqq1$.
Our convention for the basic hypergeometric series is
\[
{}_4\phi_3\!\left(
\begin{matrix}
a_1,a_2,a_3,a_4\\
b_1,b_2,b_3
\end{matrix};q,z
\right)
=\sum_{j=0}^{\infty}\frac{(a_1,a_2,a_3,a_4;q)_j}{(b_1,b_2,b_3,q;q)_j}z^j.
\]
For a pair of words $(w,w')$, set
\[
\beta_{w,w'}\coloneqq\frac{B}{A}q^{1+\#_{AB}(w)+\#_{AB}(w')},\qquad\gamma_{w,w'}\coloneqq\frac{C}{A}q^{1+\#_{AC}(w)+\#_{AC}(w')}.
\]
When the pair $(w,w')$ is fixed, we abbreviate $\beta=\beta_{w,w'}$ and $\gamma=\gamma_{w,w'}$.

We now define the connector used in this paper.
\begin{definition}[The connector]\label{def:connector}
Let $(w,w')$ be a pair of words.
For nonnegative integers $n$ and $m$ with $n+m\leq N$, we define the \emph{connector} $\scrC_{w,w'}(n,m)$ by
\begin{align*}
\scrC_{w,w'}(n, m)&\coloneqq q^{(n+\#_{\hat{A}}(w))(m+\#_{\hat{A}}(w'))}\frac{(\beta,\gamma;q)_{n+\#_{\hat{A}}(w)}(\beta,\gamma;q)_{m+\#_{\hat{A}}(w')}}{(\beta,\gamma;q)_{N+\#_{\hat{A}}(w)+\#_{\hat{A}}(w')}}\\
&\qquad\cdot\frac{(q^{n+1+\#_{\hat{A}}(w)},q^{m+1+\#_{\hat{A}}(w')},\beta\gamma q^{n+m-1+\#_{\hat{A}}(w)+\#_{\hat{A}}(w')};q)_{N-n-m}}{(q;q)_{N-n-m}}\\
&\qquad\cdot {}_4\phi_3\!\left(\begin{matrix}q^{-N+n+m},q,q/\beta,q/\gamma\\q^{n+1+\#_{\hat{A}}(w)},q^{m+1+\#_{\hat{A}}(w')},q^{2-N-\#_{\hat{A}}(w)-\#_{\hat{A}}(w')}/(\beta\gamma)\end{matrix};q,q\right).
\end{align*}
Set $\scrC_{w,w'}(n,m)\coloneqq0$ for nonnegative integers $n$, $m$ with $n+m>N$.
\end{definition}
The ${}_4\phi_3$ series above terminates at $j=N-n-m$, since $(q^{-N+n+m};q)_j=0$ for $j>N-n-m$.
The definition also immediately gives the symmetry
\begin{equation}\label{eq:sym_of_connector}
\scrC_{w,w'}(n,m)=\scrC_{w',w}(m,n).
\end{equation}
For the remainder of this section, we fix a pair of words $(w,w')$ and nonnegative integers $n, m$.
We omit some calculations in the proofs of \cref{lem:Delta,lem:long}.
Although these involve only elementary manipulations, carrying them out in full by hand is somewhat laborious.
\begin{lemma}\label{lem:Delta}
Put
\[
\Delta\scrC_{w,w'}(n,m)\coloneqq\scrC_{w,w'}(n,m)-\scrC_{w,w'}(n+1,m).
\]
If $n+m\leq N$, then
\begin{align*}
\Delta\scrC_{w,w'}(n,m)&=q^{(n+\#_{\hat{A}}(w))(m+\#_{\hat{A}}(w'))}\frac{(\beta,\gamma;q)_{n+\#_{\hat{A}}(w)}(\beta,\gamma;q)_{m+\#_{\hat{A}}(w')}
}{(\beta,\gamma;q)_{N+\#_{\hat{A}}(w)+\#_{\hat{A}}(w')}}\\&\quad\cdot\frac{\bigl(q^{n+1+\#_{\hat{A}}(w)},q^{m+\#_{\hat{A}}(w')},\beta\gamma q^{n+m+\#_{\hat{A}}(w)+\#_{\hat{A}}(w')};q\bigr)_{N-n-m}}{(q;q)_{N-n-m}}.
\end{align*}
If $n+m>N$, then $\Delta\scrC_{w,w'}(n,m)=0$.
\end{lemma}
\begin{proof}
First suppose $n+m<N$.
Let $S_j(n,m)$ denote the $j$th summand of $\scrC_{w,w'}(n,m)$, including factors preceding the series.
Thus $S_0(n,m)$ is precisely the product of those factors, and
\[
S_j(n,m)=S_0(n,m)\frac{(q^{-N+n+m},q/\beta,q/\gamma;q)_j}{(q^{n+1+\#_{\hat{A}}(w)},q^{m+1+\#_{\hat{A}}(w')},q^{2-N-\#_{\hat{A}}(w)-\#_{\hat{A}}(w')}/(\beta\gamma);q)_j}q^j.
\]
Using this, we obtain
\[
\frac{S_j(n+1,m)}{S_j(n,m)}=\frac{q^{m+\#_{\hat{A}}(w')}(q^j-q^{N-n-m})}{(1-q^{n+j+1+\#_{\hat{A}}(w)})(1-q^{N-n+\#_{\hat{A}}(w')})}\cdot\frac{(1-\beta q^{n+\#_{\hat{A}}(w)})(1-\gamma q^{n+\#_{\hat{A}}(w)})}{1-\beta\gamma q^{n+m-1+\#_{\hat{A}}(w)+\#_{\hat{A}}(w')}}
\]
and
\[
\frac{S_{j+1}(n,m)}{S_j(n,m)}=\frac{q(1-q^{-N+n+m+j})}{(1-q^{n+j+1+\#_{\hat{A}}(w)})(1-q^{m+j+1+\#_{\hat{A}}(w')})}\cdot\frac{(1-q^{j+1}/\beta)(1-q^{j+1}/\gamma)}{1-q^{j+2-N-\#_{\hat{A}}(w)-\#_{\hat{A}}(w')}/(\beta\gamma)}
\]
for $0\leq j\leq N-n-m$.
Note that although $1-q^{N-n+\#_{\hat{A}}(w')}$ may vanish when $n+m=N$,
we are assuming $n+m<N$ here.
For any $j\geq 0$, define
\[
T_j(n,m)\coloneqq S_j(n,m)\frac{(1-q^{m+j+\#_{\hat{A}}(w')})(1-\beta\gamma q^{N-1-j+\#_{\hat{A}}(w)+\#_{\hat{A}}(w')})}{(1-q^{N-n+\#_{\hat{A}}(w')})(1-\beta\gamma q^{n+m-1+\#_{\hat{A}}(w)+\#_{\hat{A}}(w')})}.
\]
Using the two ratio formulas above, we verify the following $q$-WZ relation:
\[
S_j(n,m)-S_j(n+1,m)=T_j(n,m)-T_{j+1}(n,m).
\]
Since $S_j(n,m)=0$ for $j>N-n-m$, we also have $T_{N-n-m+1}(n,m)=0$.
Summing the $q$-WZ relation over $j=0,\ldots,N-n-m$ gives
\[
\Delta\scrC_{w,w'}(n,m)=T_0(n,m)=S_0(n,m)\frac{(1-q^{m+\#_{\hat{A}}(w')})(1-\beta\gamma q^{N-1+\#_{\hat{A}}(w)+\#_{\hat{A}}(w')})}{(1-q^{N-n+\#_{\hat{A}}(w')})(1-\beta\gamma q^{n+m-1+\#_{\hat{A}}(w)+\#_{\hat{A}}(w')})}.
\]
Since 
\[
(q^{m+1+\#_{\hat{A}}(w')};q)_{N-n-m}\cdot\frac{1-q^{m+\#_{\hat{A}}(w')}}{1-q^{N-n+\#_{\hat{A}}(w')}}=(q^{m+\#_{\hat A}(w')};q)_{N-n-m}
\]
and 
\begin{align*}
&(\beta\gamma q^{n+m-1+\#_{\hat{A}}(w)+\#_{\hat{A}}(w')};q)_{N-n-m}\cdot\frac{1-\beta\gamma q^{N-1+\#_{\hat{A}}(w)+\#_{\hat{A}}(w')}}{1-\beta\gamma q^{n+m-1+\#_{\hat{A}}(w)+\#_{\hat{A}}(w')}}\\
&=(\beta\gamma q^{n+m+\#_{\hat{A}}(w)+\#_{\hat{A}}(w')};q)_{N-n-m}
\end{align*}
hold, we obtain the desired formula.
For $n+m\geq N$, the assertion follows immediately from the definition.
\end{proof}
We define $A^{\circ}, B^{\circ}, C^{\circ}, D^{\circ}$, $A^{\bullet}, B^{\bullet}, C^{\bullet}, D^{\bullet}$, which depend on $(w, w')$ (but not on the moving letter $uv$ introduced below), by
\[
\begin{alignedat}{2}
A^\circ&\coloneqq Aq^{\#_{\hat A}(w)},
&\qquad
A^\bullet&\coloneqq Aq^{\#_{\hat A}(w')},\\
B^\circ&\coloneqq Bq^{1+\#_{\hat A}(w)+\#_{AB}(w)+\#_{AB}(w')},
&\qquad
B^\bullet&\coloneqq Bq^{1+\#_{\hat A}(w')+\#_{AB}(w)+\#_{AB}(w')},\\
C^\circ&\coloneqq Cq^{1+\#_{\hat A}(w)+\#_{AC}(w)+\#_{AC}(w')},
&\qquad
C^\bullet&\coloneqq Cq^{1+\#_{\hat A}(w')+\#_{AC}(w)+\#_{AC}(w')},\\
D^\circ&\coloneqq Dq^{-\#_{\hat A}(w')},
&\qquad
D^\bullet &\coloneqq Dq^{-\#_{\hat A}(w)}.
\end{alignedat}
\]
A pair $(w,w')$ is \emph{admissible} if either both words are nonempty and each starts with $BC$, $BD$, or $CD$, or one word is empty and the other is admissible.
This is equivalent to the condition that $w\tau(w')$ is an admissible word.
\begin{lemma}\label{lem:long}
Let $uv\in\{AB,AC,AD,BC,BD,CD\}$, write $\tau(uv)=\tilde{u}\tilde{v}$, and assume that $(wuv,w')$ is an admissible pair.
For nonnegative integers $i,j$ with $i+j<N$, we have
\begin{equation}\label{eq:double_difference}
\begin{split}
&\left(\frac{Aq^{-i}}{Aq^{-i}-u^{\circ}}-\frac{Aq^{-i}}{Aq^{-i}-v^{\circ}}\right)^{-1}\Delta\scrC_{w,w'\tau(uv)}(i,j)\\
&\quad-\left(\frac{Aq^{-i-1}}{Aq^{-i-1}-u^{\circ}}-\frac{Aq^{-i-1}}{Aq^{-i-1}-v^{\circ}}\right)^{-1}\Delta\scrC_{w,w'\tau(uv)}(i+1,j)\\
&=\left(\frac{Aq^{-j}}{Aq^{-j}-\tilde{u}^{\bullet}}-\frac{Aq^{-j}}{Aq^{-j}-\tilde{v}^{\bullet}}\right)^{-1}\left(\Delta\scrC_{wuv,w'}(i,j)-\Delta\scrC_{wuv,w'}(i,j+1)\right).
\end{split}
\end{equation}
If $i+j=N$, we have instead the boundary identity
\[
\left(\frac{Aq^{-i}}{Aq^{-i}-u^\circ}-\frac{Aq^{-i}}{Aq^{-i}-v^\circ}
\right)^{-1}\Delta\scrC_{w,w'\tau(uv)}(i,j)=\left(\frac{Aq^{-j}}{Aq^{-j}-\tilde{u}^\bullet}-\frac{Aq^{-j}}{Aq^{-j}-\tilde{v}^\bullet}\right)^{-1}\Delta\scrC_{wuv,w'}(i,j).
\]
\end{lemma}
\begin{proof}
We introduce the following quantities, which may depend on the moving letter $uv$ as well as on $(w, w')$:
\begin{align*}
A^+&\coloneqq A^{\circ}q^{\#_{\hat{A}}(uv)}=Aq^{\#_{\hat{A}}(wuv)},\\
B^+&\coloneqq B^{\circ}q^{\#_{\hat{A}}(uv)+\#_{AB}(uv)}=Aq^{\#_{\hat{A}}(wuv)}\beta_{wuv,w'},\\
C^+&\coloneqq C^{\circ}q^{\#_{\hat{A}}(uv)+\#_{AC}(uv)}=Aq^{\#_{\hat{A}}(wuv)}\gamma_{wuv,w'},\\
D^+&\coloneqq D^{\circ},
\end{align*}
and $\theta\coloneqq\frac{A^+D^+}{B^+C^+}=q^{\#_{AD}(uv)-1}\frac{A^{\circ}D^{\circ}}{B^{\circ}C^{\circ}}$.
We first consider the case $i+j<N$.
Since $(w,w'\tau(uv))$ is admissible and $w'\tau(uv)$ is nonempty, we have $\#_{\hat A}(w'\tau(uv))\geq1$.
Thus \cref{lem:Delta} and our algebraic independence assumptions imply that $\Delta\scrC_{w,w'\tau(uv)}(i,j)\neq0$.
The moving factors are also well defined and nonzero.
Applying \cref{lem:Delta} separately to the pairs $(w,w'\tau(uv))$ and $(wuv,w')$, and simplifying the resulting $q$-shifted factorials, gives the following three identities:
\begin{equation}\label{eq:ratio1}
\begin{split}
&\frac{(Aq^{-i}-qu^{\circ})(Aq^{-i}-qv^{\circ})}{q(Aq^{-i}-u^{\circ})(Aq^{-i}-v^{\circ})}\cdot\frac{\Delta\scrC_{w,w'\tau(uv)}(i+1,j)}{\Delta\scrC_{w,w'\tau(uv)}(i,j)}\\
&=\frac{\theta(q^{N-i-j}-1)(Aq^{-i}-B^+)(Aq^{-i}-C^+)}{(\theta q^{N-i-j}-1)(Aq^{-i}-A^+)(Aq^{-i}-D^+)},
\end{split}
\end{equation}
\begin{equation}\label{eq:ratio2}
\begin{split}
&q^{j-i}\cdot\frac{u^{\circ}-v^{\circ}}{\tilde{u}^{\bullet}-\tilde{v}^{\bullet}}\cdot\frac{(Aq^{-j}-\tilde{u}^{\bullet})(Aq^{-j}-\tilde{v}^{\bullet})}{(Aq^{-i}-u^{\circ})(Aq^{-i}-v^{\circ})}\cdot\frac{\Delta\scrC_{wuv,w'}(i,j)}{\Delta\scrC_{w,w'\tau(uv)}(i,j)}\\
&=\frac{(Aq^{-i}-A^+q^{N-i-j})(Aq^{-i}-D^+q^{N-i-j})}{q^{N-i-j}(Aq^{-i}-A^+)(Aq^{-i}-D^+)},
\end{split}
\end{equation}
and
\begin{equation}\label{eq:ratio3}
\begin{split}
&q^{j-i}\cdot\frac{u^{\circ}-v^{\circ}}{\tilde{u}^{\bullet}-\tilde{v}^{\bullet}}\cdot\frac{(Aq^{-j}-\tilde{u}^{\bullet})(Aq^{-j}-\tilde{v}^{\bullet})}{(Aq^{-i}-u^{\circ})(Aq^{-i}-v^{\circ})}\cdot\frac{\Delta\scrC_{wuv,w'}(i,j+1)}{\Delta\scrC_{w,w'\tau(uv)}(i,j)}\\
&=\frac{(q^{N-i-j}-1)(Aq^{-i}-\theta B^+q^{N-i-j})(Aq^{-i}-\theta C^+q^{N-i-j})}{q^{N-i-j}(\theta q^{N-i-j}-1)(Aq^{-i}-A^+)(Aq^{-i}-D^+)}.
\end{split}
\end{equation}
Each identity follows by elementary cancellation of $q$-shifted factorials in \cref{lem:Delta}.
Although the parameter shifts depend on the moving letter, the resulting identities have the same form for all six choices.
For illustration, we verify \eqref{eq:ratio2} in the case $uv=AB$.
In this case, $\tau(AB)=CD$.
Writing $\beta=\beta_{w,w'}$ and $\gamma=\gamma_{w,w'}$, we obtain from \cref{lem:Delta}
\begin{align*}
\Delta\scrC_{wAB,w'}(i,j)&=q^{(i+\#_{\hat{A}}(w))(j+\#_{\hat{A}}(w'))}\frac{(\beta q,\gamma;q)_{i+\#_{\hat{A}}(w)}(\beta q,\gamma;q)_{j+\#_{\hat{A}}(w')}
}{(\beta q,\gamma;q)_{N+\#_{\hat{A}}(w)+\#_{\hat{A}}(w')}}\\&\quad\cdot\frac{\bigl(q^{i+1+\#_{\hat{A}}(w)},q^{j+\#_{\hat{A}}(w')},\beta\gamma q^{i+j+1+\#_{\hat{A}}(w)+\#_{\hat{A}}(w')};q\bigr)_{N-i-j}}{(q;q)_{N-i-j}}
\end{align*}
and
\begin{align*}
\Delta\scrC_{w,w'CD}(i,j)&=q^{(i+\#_{\hat{A}}(w))(j+1+\#_{\hat{A}}(w'))}\frac{(\beta,\gamma;q)_{i+\#_{\hat{A}}(w)}(\beta,\gamma;q)_{j+1+\#_{\hat{A}}(w')}
}{(\beta,\gamma;q)_{N+1+\#_{\hat{A}}(w)+\#_{\hat{A}}(w')}}\\&\quad\cdot\frac{\bigl(q^{i+1+\#_{\hat{A}}(w)},q^{j+1+\#_{\hat{A}}(w')},\beta\gamma q^{i+j+1+\#_{\hat{A}}(w)+\#_{\hat{A}}(w')};q\bigr)_{N-i-j}}{(q;q)_{N-i-j}}.
\end{align*}
Dividing these expressions and cancelling the common factors gives
\[
\frac{\Delta\scrC_{wAB,w'}(i,j)}{\Delta\scrC_{w,w'CD}(i,j)}=q^{-i-\#_{\hat{A}}(w)}\cdot\frac{(1-\beta q^{i+\#_{\hat{A}}(w)})(1-\gamma q^{N+\#_{\hat{A}}(w)+\#_{\hat{A}}(w')})}{(1-\beta)(1-\gamma q^{j+\#_{\hat{A}}(w')})}\cdot\frac{1-q^{j+\#_{\hat{A}}(w')}}{1-q^{N-i+\#_{\hat{A}}(w')}}.
\]
Using
\[
\beta=\frac{B^{\circ}}{{A^{\circ}}}, \quad \gamma=\frac{C^{\circ}}{A^{\circ}}, \quad \theta=\frac{A^{\circ}D^{\circ}}{qB^{\circ}C^{\circ}}
\]
together with
\[
q^{i+\#_{\hat{A}}(w)}=\frac{A^{\circ}}{Aq^{-i}}, \quad q^{j+\#_{\hat{A}}(w')}=\frac{Aq^{-i}}{D^{\circ}q^{N-i-j}}, \quad q^{N+\#_{\hat{A}}(w)+\#_{\hat{A}}(w')}=\frac{A^{\circ}}{D^{\circ}},
\]
we can rewrite this ratio as
\[
\frac{\Delta\scrC_{wAB,w'}(i,j)}{\Delta\scrC_{w,w'CD}(i,j)}=\frac{A^{\circ}(C^{\circ}-D^{\circ})(Aq^{-i}-B^{\circ})(Aq^{-i}-D^{\circ}q^{N-i-j})}{C^{\circ}(A^{\circ}-B^{\circ})(Aq^{-i}-D^{\circ})(Aq^{-i}-qB^{\circ}\theta q^{N-i-j})}.
\]
For any moving letter $uv$,
\begin{align*}
&q^{j-i}\cdot\frac{u^{\circ}-v^{\circ}}{\tilde{u}^{\bullet}-\tilde{v}^{\bullet}}\cdot\frac{(Aq^{-j}-\tilde{u}^{\bullet})(Aq^{-j}-\tilde{v}^{\bullet})}{(Aq^{-i}-u^{\circ})(Aq^{-i}-v^{\circ})}\\
&=\frac{\tilde{u}^{\bullet}\tilde{v}^{\bullet}(u^{\circ}-v^{\circ})}{AD(\tilde{u}^{\bullet}-\tilde{v}^{\bullet})}\cdot\frac{(Aq^{-i}-ADq^{N-i-j}/\tilde{u}^{\bullet})(Aq^{-i}-ADq^{N-i-j}/\tilde{v}^{\bullet})}{q^{N-i-j}(Aq^{-i}-u^{\circ})(Aq^{-i}-v^{\circ})}
\end{align*}
holds.
In the present case, we have
\[
\frac{AD}{C^{\bullet}}=qB^{\circ}\theta, \quad \frac{AD}{D^{\bullet}}=A^{\circ}.
\]
Consequently,
\[
q^{j-i}\cdot\frac{A^{\circ}-B^{\circ}}{C^{\bullet}-D^{\bullet}}\cdot\frac{(Aq^{-j}-C^{\bullet})(Aq^{-j}-D^{\bullet})}{(Aq^{-i}-A^{\circ})(Aq^{-i}-B^{\circ})}=\frac{(A^{\circ}-B^{\circ})(Aq^{-i}-A^{\circ}q^{N-i-j})(Aq^{-i}-qB^{\circ}\theta q^{N-i-j})}{q^{N-i-j}(A^{\circ}-qB^{\circ}\theta)(Aq^{-i}-A^{\circ})(Aq^{-i}-B^{\circ})}.
\]
Multiplying the two ratios, we obtain
\[
q^{j-i}\cdot\frac{A^{\circ}-B^{\circ}}{C^{\bullet}-D^{\bullet}}\cdot\frac{(Aq^{-j}-C^{\bullet})(Aq^{-j}-D^{\bullet})}{(Aq^{-i}-A^{\circ})(Aq^{-i}-B^{\circ})}\cdot\frac{\Delta\scrC_{wAB,w'}(i,j)}{\Delta\scrC_{w,w'CD}(i,j)}=\frac{(Aq^{-i}-A^{\circ}q^{N-i-j})(Aq^{-i}-D^{\circ}q^{N-i-j})}{q^{N-i-j}(Aq^{-i}-A^{\circ})(Aq^{-i}-D^{\circ})}.
\]
Since $A^+=A^\circ$ and $D^+=D^\circ$, this is precisely the second identity for $uv=AB$.
(As explained in \cref{rem}, no separate calculations are needed for the cases $uv=BD$ and $uv=CD$.)

We now return to an arbitrary moving letter $uv$.
Dividing \eqref{eq:double_difference} by
\[
\left(\frac{Aq^{-i}}{Aq^{-i}-u^\circ}-\frac{Aq^{-i}}{Aq^{-i}-v^\circ}\right)^{-1}\Delta\scrC_{w,w'\tau(uv)}(i,j),
\]
we see that it is equivalent to
\begin{align*}
&1-\frac{(Aq^{-i}-qu^{\circ})(Aq^{-i}-qv^{\circ})}{q(Aq^{-i}-u^{\circ})(Aq^{-i}-v^{\circ})}\cdot\frac{\Delta\scrC_{w,w'\tau(uv)}(i+1,j)}{\Delta\scrC_{w,w'\tau(uv)}(i,j)}\\
&=q^{j-i}\cdot\frac{u^{\circ}-v^{\circ}}{\tilde{u}^{\bullet}-\tilde{v}^{\bullet}}\cdot\frac{(Aq^{-j}-\tilde{u}^{\bullet})(Aq^{-j}-\tilde{v}^{\bullet})}{(Aq^{-i}-u^{\circ})(Aq^{-i}-v^{\circ})}\left(\frac{\Delta\scrC_{wuv,w'}(i,j)}{\Delta\scrC_{w,w'\tau(uv)}(i,j)}-\frac{\Delta\scrC_{wuv,w'}(i,j+1)}{\Delta\scrC_{w,w'\tau(uv)}(i,j)}\right).
\end{align*}
Substituting \eqref{eq:ratio1}, \eqref{eq:ratio2}, and \eqref{eq:ratio3}, it remains to prove
\begin{align*}
&1-\frac{\theta(q^{N-i-j}-1)(Aq^{-i}-B^+)(Aq^{-i}-C^+)}{(\theta q^{N-i-j}-1)(Aq^{-i}-A^+)(Aq^{-i}-D^+)}\\
&=\frac{(Aq^{-i}-A^+q^{N-i-j})(Aq^{-i}-D^+q^{N-i-j})}{q^{N-i-j}(Aq^{-i}-A^+)(Aq^{-i}-D^+)}\\
&\qquad -\frac{(q^{N-i-j}-1)(Aq^{-i}-\theta B^+q^{N-i-j})(Aq^{-i}-\theta C^+q^{N-i-j})}{q^{N-i-j}(\theta q^{N-i-j}-1)(Aq^{-i}-A^+)(Aq^{-i}-D^+)}.
\end{align*}
Put $x\coloneqq q^{N-i-j}$, $t\coloneqq Aq^{-i}$.
After multiplying by $x(t-A^+)(t-D^+)$ and rearranging, the last identity is equivalent to
\begin{align*}
&x(t-A^+)(t-D^+)-(t-A^+x)(t-D^+x)\\
&=\frac{x-1}{\theta x-1}\cdot\left(\theta x(t-B^+)(t-C^+)-(t-\theta B^+x)(t-\theta C^+x)\right).
\end{align*}
The left-hand side satisfies
\[
x(t-A^+)(t-D^+)-(t-A^+x)(t-D^+x)=(x-1)(t^2-A^+D^+x),
\]
whereas, since $\theta B^+C^+=A^+D^+$,
\[
\theta x(t-B^+)(t-C^+)-(t-\theta B^+x)(t-\theta C^+x)=(\theta x-1)(t^2-A^+D^+x).
\]
This proves \eqref{eq:double_difference} when $i+j<N$.

It remains to consider the boundary case $i+j=N$.
In this case, substituting the formula of \cref{lem:Delta} directly and cancelling the common factors gives
\[
q^{j-i}\cdot\frac{u^\circ-v^\circ}{\tilde{u}^\bullet-\tilde{v}^\bullet}\cdot\frac{(Aq^{-j}-\tilde{u}^\bullet)(Aq^{-j}-\tilde{v}^\bullet)}{(Aq^{-i}-u^\circ)(Aq^{-i}-v^\circ)}\cdot\frac{\Delta\scrC_{wuv,w'}(i,j)}{\Delta\scrC_{w,w'\tau(uv)}(i,j)}=1.
\]
This is equivalent to the required boundary identity.
\end{proof}
\begin{proposition}\label{prop:local_transport_rel}
Let $uv\in\{AB,AC,AD,BC,BD,CD\}$, write $\tau(uv)=\tilde{u}\tilde{v}$, and assume that $(wuv,w')$ is an admissible pair.
If $n+m\leq N$, then
\begin{equation}\label{eq:local_transport_rel}
\begin{split}
&\sum_{a=n}^{N-m}\left(\frac{Aq^{-a}}{Aq^{-a}-u^{\circ}}-\frac{Aq^{-a}}{Aq^{-a}-v^{\circ}}\right)\scrC_{wuv,w'}(a,m)\\
&=\sum_{b=m}^{N-n}\scrC_{w,w'\tau(uv)}(n,b)\left(\frac{Aq^{-b}}{Aq^{-b}-\tilde{u}^{\bullet}}-\frac{Aq^{-b}}{Aq^{-b}-\tilde{v}^{\bullet}}\right).
\end{split}
\end{equation}
\end{proposition}
\begin{proof}
Fix a nonnegative integer $i$ with $i+m<N$.
For $b=m,\ldots,N-i-1$, multiply \eqref{eq:double_difference} with $j=b$ by
\[
\frac{Aq^{-b}}{Aq^{-b}-\tilde{u}^{\bullet}}-\frac{Aq^{-b}}{Aq^{-b}-\tilde{v}^{\bullet}}.
\]
Summing these identities and adding the boundary identity in \cref{lem:long}, multiplied by the same factor with $b=N-i$, we obtain
\begin{align*}
&\left(\frac{Aq^{-i}}{Aq^{-i}-u^{\circ}}-\frac{Aq^{-i}}{Aq^{-i}-v^{\circ}}\right)^{-1}\sum_{b=m}^{N-i}\Delta\scrC_{w,w'\tau(uv)}(i,b)\left(\frac{Aq^{-b}}{Aq^{-b}-\tilde{u}^{\bullet}}-\frac{Aq^{-b}}{Aq^{-b}-\tilde{v}^{\bullet}}\right)\\
&\quad-\left(\frac{Aq^{-i-1}}{Aq^{-i-1}-u^{\circ}}-\frac{Aq^{-i-1}}{Aq^{-i-1}-v^{\circ}}\right)^{-1}\sum_{b=m}^{N-i-1}\Delta\scrC_{w,w'\tau(uv)}(i+1,b)\left(\frac{Aq^{-b}}{Aq^{-b}-\tilde{u}^{\bullet}}-\frac{Aq^{-b}}{Aq^{-b}-\tilde{v}^{\bullet}}\right)\\
&=\Delta\scrC_{wuv,w'}(i,m).
\end{align*}
For $i=N-m$, the same identity holds with the second term omitted, by the boundary identity in \cref{lem:long}.

Now fix $a$ with $n\leq a\leq N-m$ and sum the above identity over $i=a,\ldots,N-m$.
Then we have
\[
\sum_{b=m}^{N-a}\Delta\scrC_{w,w'\tau(uv)}(a,b)\left(\frac{Aq^{-b}}{Aq^{-b}-\tilde{u}^{\bullet}}-\frac{Aq^{-b}}{Aq^{-b}-\tilde{v}^{\bullet}}\right)=\left(\frac{Aq^{-a}}{Aq^{-a}-u^{\circ}}-\frac{Aq^{-a}}{Aq^{-a}-v^{\circ}}\right)\scrC_{wuv,w'}(a,m).
\]
Finally, summing over $a=n,\ldots,N-m$ and interchanging the two finite sums, we obtain
\begin{align*}
&\sum_{a=n}^{N-m}\sum_{b=m}^{N-a}\Delta\scrC_{w,w'\tau(uv)}(a,b)\left(\frac{Aq^{-b}}{Aq^{-b}-\tilde{u}^{\bullet}}-\frac{Aq^{-b}}{Aq^{-b}-\tilde{v}^{\bullet}}\right)\\
&=\sum_{b=m}^{N-n}\left(\sum_{a=n}^{N-b}\Delta\scrC_{w,w'\tau(uv)}(a,b)\right)\left(\frac{Aq^{-b}}{Aq^{-b}-\tilde{u}^{\bullet}}-\frac{Aq^{-b}}{Aq^{-b}-\tilde{v}^{\bullet}}\right)\\
&=\sum_{b=m}^{N-n}\scrC_{w,w'\tau(uv)}(n,b)\left(\frac{Aq^{-b}}{Aq^{-b}-\tilde{u}^{\bullet}}-\frac{Aq^{-b}}{Aq^{-b}-\tilde{v}^{\bullet}}\right).
\end{align*}
Together with the preceding identity, this gives the desired formula.
\end{proof}
\begin{remark}\label{rem}
The cases $uv=BD$ and $uv=CD$ of \cref{lem:long} need not be verified directly.
Indeed, first suppose that \cref{lem:long} has been proved for $uv\in\{AB,AC,AD,BC\}$.
For each of these four letters, the proof of \cref{prop:local_transport_rel} gives \eqref{eq:local_transport_rel}.
Interchanging $(w,n)$ and $(w',m)$ and using \eqref{eq:sym_of_connector}, together with the interchange of the $\circ$- and $\bullet$-parameters, gives the identity for $\tau(uv)$.
Thus \eqref{eq:local_transport_rel} holds for all six letters.
Conversely, for a fixed moving letter $uv$, taking successive differences in the summation bounds reverses the argument in the proof of \cref{prop:local_transport_rel} and recovers \eqref{eq:double_difference}.
When $n+m=N$, both sums in \eqref{eq:local_transport_rel} consist of a single term, giving the boundary identity in \cref{lem:long}.
This proves the remaining two cases of the lemma.
\end{remark}
\section{The connected sum}
In this section, we define the connected sum using the connector introduced in the previous section and prove the main theorem.
\begin{definition}[The connected sum]\label{def:connected_sum}
For two words $w=u_1v_1\dots u_kv_k$, $w'=u_1'v_1'\cdots u_l'v_l'$, and $i\in[k]$, put
\begin{align*}
A^{(i),w,w'}&\coloneqq A^{(i),w},\\
B^{(i),w,w'}&\coloneqq B^{(i),w}q^{\#_{AB}(w')},\\
C^{(i),w,w'}&\coloneqq C^{(i),w}q^{\#_{AC}(w')},\\
D^{(i),w,w'}&\coloneqq D^{(i),w}q^{-\#_{\hat{A}}(w')}.
\end{align*}
When $(w,w')$ is an admissible pair, we define the \emph{connected sum} $Z(w,w')$ by
\begin{align*}
Z(w, w')\coloneqq&\sum_{\substack{0=n_0\leq n_1\leq \cdots \leq n_k\leq N \\ 0=m_0\leq m_1\leq \cdots \leq m_l\leq N}}\prod_{i=1}^k\left(\frac{Aq^{-n_i}}{Aq^{-n_i}-u_i^{(i),w,w'}}-\frac{Aq^{-n_i}}{Aq^{-n_i}-v_i^{(i),w,w'}}\right)\\
&\cdot \scrC_{w,w'}(n_k, m_l)\cdot\prod_{j=1}^l\left(\frac{Aq^{-m_j}}{Aq^{-m_j}-{u'_j}^{(j),w',w}}-\frac{Aq^{-m_j}}{Aq^{-m_j}-{v'_j}^{(j),w',w}}\right).
\end{align*}
Empty products are understood to be $1$.
\end{definition}
\begin{lemma}[Symmetry]
For every admissible pair $(w,w')$, we have
\[
Z(w,w')=Z(w',w).
\]
\end{lemma}
\begin{proof}
This follows by interchanging the two sets of summation indices and using the symmetry \eqref{eq:sym_of_connector} of the connector.
\end{proof}
\begin{lemma}[The boundary condition]\label{lem:boundary_condition}
For every admissible word $w$, we have
\[
Z(w,1)=Z(1,w)=L_q(w).
\]
\end{lemma}
\begin{proof}
It suffices to show that $\scrC_{w,1}(n,0)=1$ for $0\leq n\leq N$.
The definition gives $\scrC_{w,1}(N,0)=1$.
For $0\leq n<N$, by \cref{lem:Delta} and $(q^{0+\#_{\hat A}(1)};q)_{N-n}=0$, we have $\Delta\scrC_{w,1}(n,0)=0$.
Hence $\scrC_{w,1}(n,0)=\scrC_{w,1}(n+1,0)=\cdots=\scrC_{w,1}(N,0)=1$.
\end{proof}
\begin{proposition}[The transport relation]\label{prop:connected_sum_transport_rel}
Let $uv$ be a single letter, and assume that $(wuv,w')$ is an admissible pair.
Then we have
\[
Z(wuv,w')=Z(w,w'\tau(uv)).
\]
\end{proposition}
\begin{proof}
Let $k$ and $l$ be the lengths of $w$ and $w'$, respectively, and write $\tau(uv)=\tilde{u}\tilde{v}$.
For every $i\in[k]$ and $X\in\{A,B,C,D\}$, we have $X^{(i),wuv,w'}=X^{(i),w,w'\tau(uv)}$.
For example,
\[
B^{(i),wuv,w'}=B^{(i),w,w'}q^{\#_{CD}(uv)}=B^{(i),w,w'}q^{\#_{AB}(\tau(uv))}=B^{(i),w,w'\tau(uv)}.
\]
Similarly, for every $j\in[l]$ and $X\in\{A,B,C,D\}$, $X^{(j),w',wuv}=X^{(j),w'\tau(uv),w}$.
Thus all factors corresponding to the letters of $w$ and $w'$ are the same in the two connected sums.

Next, with the notation of \cref{prop:local_transport_rel}, we have
\[
u^{(k+1),wuv,w'}=u^{\circ}, \quad v^{(k+1),wuv,w'}=v^{\circ}, \quad \tilde{u}^{(l+1),w'\tau(uv),w}=\tilde{u}^{\bullet},\quad \tilde{v}^{(l+1),w'\tau(uv),w}=\tilde{v}^{\bullet}.
\]
For example, if $B\in\{u,v\}$, then $uv\in\{AB,BC,BD\}$.
Hence $\#_{\hat{A}}(uv)+\#_{AB}(uv)=1$ and $\#_{CD}(uv)=0$.
It follows from the definitions that
\[
B^{(k+1),wuv,w'}=Bq^{\#_{\hat{A}}(wuv)+\#_{AB}(wuv)+\#_{CD}(uv)+\#_{AB}(w')}=Bq^{1+\#_{\hat{A}}(w)+\#_{AB}(w)+\#_{AB}(w')}=B^\circ.
\]

Now fix the summation indices belonging to $w$ and $w'$:
\[
0\leq n_1\leq\cdots\leq n_k\leq N,
\qquad
0\leq m_1\leq\cdots\leq m_l\leq N.
\]
For an empty word, we use the convention $n_0=m_0=0$.
If $n_k+m_l>N$, the corresponding contributions to both connected sums vanish, since the connector is zero when the sum of its arguments exceeds $N$.
Otherwise, \cref{prop:local_transport_rel} gives
\begin{align*}
&\sum_{n_{k+1}=n_k}^{N-m_l}\left(\frac{Aq^{-n_{k+1}}}{Aq^{-n_{k+1}}-u^{(k+1),wuv,w'}}-\frac{Aq^{-n_{k+1}}}{Aq^{-n_{k+1}}-v^{(k+1),wuv,w'}}\right)\scrC_{wuv,w'}(n_{k+1},m_l)\\
&=\sum_{m_{l+1}=m_l}^{N-n_k}\scrC_{w,w'\tau(uv)}(n_k,m_{l+1})\left(\frac{Aq^{-m_{l+1}}}{Aq^{-m_{l+1}}-\tilde{u}^{(l+1),w'\tau(uv), w}}-\frac{Aq^{-m_{l+1}}}{Aq^{-m_{l+1}}-\tilde{v}^{(l+1),w'\tau(uv), w}}\right).
\end{align*}
Multiplying this identity by the common factors corresponding to the letters of $w$ and $w'$, and summing over all choices of $n_1,\ldots,n_k,m_1,\ldots,m_l$, yields the desired formula.
\end{proof}
\begin{proof}[Proof of \cref{thm:main}]
Let $w$ be an admissible word of length $k$.
We repeatedly apply \cref{prop:connected_sum_transport_rel} to move the last letter of the left word to the right.
All intermediate pairs are admissible.
Together with \cref{lem:boundary_condition}, this gives
\begin{align*}
L_q(w)&=Z(w,1)=Z(w_{[k-1]},\tau(w^{[k]}))=Z(w_{[k-2]},\tau(w^{[k-1]}))=\cdots=Z(w_{[1]},\tau(w^{[2]}))\\
&=Z(1,\tau(w))=L_q(\tau(w)).
\end{align*}
(Taken literally, the displayed chain applies only when $k\geq2$; for $k=0,1$, the same conclusion follows from the corresponding shorter chain.)
\end{proof}
\subsection*{Acknowledgement}
In mathematics, the discovery of a problem or the formulation of a conjecture is often an achievement of greater significance than its resolution.
This is all the more worth emphasizing in an era when generative AI is accelerating the pace at which open problems are resolved.

Even before the preprint~\cite{Hirose} appeared, Professor Minoru Hirose
had spoken to the author about his duality conjecture.
He had described its formulation in terms of six letters and the difficulty of proving it, despite its being a generalization of the classical duality.
The author wishes to express his deep appreciation to Professor Hirose for bringing this beautiful mathematical phenomenon to light.
\subsection*{Use of AI}
ChatGPT, powered by GPT-6 Astra Pro, was used for mathematical discussions and to refine the English prose of the manuscript.
The final manuscript and all mathematical arguments were written and independently verified by the author, who takes full responsibility for the content.

\end{document}